\documentclass[preprint,12pt]{elsarticle}
\usepackage[margin=1in]{geometry}
\usepackage{setspace}
\usepackage{adjustbox}
\usepackage{array}
\usepackage{mathrsfs}
\usepackage[normalem]{ulem}
\usepackage{dirtytalk}
\usepackage{natbib}
\setcitestyle{authoryear,open={(},close={)}}
\usepackage{pgfplots}
\usepackage{multirow}
\usepackage{lscape} 
\usepackage{arydshln}
\usepackage{tikz}
\usetikzlibrary{arrows.meta}
\usepackage{pifont}
\usepackage{enumitem}

\usepackage{color, colortbl}
\usepackage{pgfplots}
\usepackage{pgfplotstable}
\definecolor{Gray}{gray}{0.9}
\usepackage{float}
\usetikzlibrary[positioning]
\usetikzlibrary{patterns}
\usepackage[toc,page]{appendix}
\usepackage{tkz-euclide}
\usepackage{algorithm,amsmath,lipsum,xcolor,caption}
\usepackage{graphicx}
\usepackage{caption}
\usepackage{subcaption}
\usepackage{setspace}
\usepackage{forest}
\usepackage{xcolor,colortbl}
\usepackage{arydshln}
\usepackage{amssymb}
\usepackage{amsthm}
\usepackage{url}
\usepackage{eurosym}
\usepackage{censor}

\newtheorem{theorem}{Theorem}
\newtheorem{lemma}{Lemma}

\newtheorem{remark}{Remark}

\newtheorem{example}{Example}
\newtheorem*{proof*}{Proof}
\begin{document}

\begin{frontmatter}

\title{Fighting Pickpocketing using  a Choice-based \\ Resource Allocation Model}

\author[tue]{Loe Schlicher\corref{cor1}}
\ead{}
\address[tue]{Eindhoven University of Technology, Department of Industrial Engineering and Innovation Sciences \\ P.O. Box 513, 5600 MB, Eindhoven, The Netherlands}

\author[hec]{Virginie Lurkin} 
\ead{}
\cortext[cor1]{Corresponding author}
\address[hec]{University of Lausanne, HEC Lausanne Faculty of Business and Economics \\ Quartier de Chamberonne CH-1015, Lausanne, Switzerland}


\begin{abstract}
Inspired by European actions to fight organised crimes, we develop a choice-based resource allocation model that can help policy makers to reduce the number of pickpocket attempts. In this model, the policy maker needs to allocate a limited budget over local and central protective resources as well as over potential pickpocket locations, while keeping in mind the thieves' preferences towards potential pickpocket locations.  We prove that the optimal budget allocation is proportional in ($i$) the thieves' sensitivity towards protective resources and ($ii$) the initial attractiveness of the potential pickpocket locations. By means of a numerical experiment, we illustrate how this optimal budget allocation performs against various others budget allocations, proposed by policy makers from the field. \end{abstract}


\end{frontmatter}

\section{Introduction}

According to the European Union (EU) Agency for Law Enforcement, the threat posed by organised crime in Europe has never been higher. In the latest Serious and Organised Crime Threat Assessment report, it is estimated that the annual impact of organised crime on the EU economy falls between \euro 218 billion and \euro 282 billion in terms of gross domestic product (\citet{luyten2020understanding}). Another example of the financial impact of organized crimes comes from the Norwegian police that, in 2018, estimated the combined value of stolen goods to be around \euro 2.3 million, only for the city center of Oslo (\citet{Policenorway}). In order to fight these crimes, a large number of European and national actions have been undertaken over the last decade (\citet{EUCPN2}). A striking collaborative initative is the European Harmony project (\citet{Frans}), led by Belgium with Europol and including the Netherlands and the United Kingdom working as partners, whose aim was to produce a European policy cycle regarding action against organised and serious transnational crime.

A well-known and highly visible form of organised crime is pickpocketing. It refers to the unnoticed theft of items that are carried on the body of the victim (\citet{EUCPN2}). European city centers often present a multitude of criminal opportunities for (pickpocket) thieves\footnote{In this paper, we will use the term "thief/thieves" to refer to someone who has plans to  pickpocket.}, mainly due to the high concentration of tourists exploring crowded European cities every year (\citet{SOCTA2021}). In the period 2009-2015, more than 52 million attempts of pickpocketing have been reported in Europe (\citet{EUCPN}). Next to the feeling of insecurity, pickpocketing leads to significant costs on an individual level (e.g., victims need to replace stolen items) and on a national level (e.g., costs by public authorities for detection, prevention and prosecution). To reduce pickpocketing odds, awareness campaigns are typically conducted. An example of a national campaign is Boefproof in the Netherlands that encourages citizens to activate or install an anti-feature on a mobile device, allowing to make it non-accessible by remote control, and thus making it worthless for thieves in case it gets stolen. Boefproof is part of a larger campaign of the Dutch Ministry of Security and Justice, which aims at further enhancing the awareness of citizens about pickpocketing through the use of billboards, radio commercials or social media posts (e.g., via Facebook or Twitter), even using influencers  (\citet{ccv}). Awareness campaigns are also adopted in other European countries, in the form of radio commercials in Prague (Czech Republic), short videos broadcasted in hotel rooms in Brussels (Belgium) and warning templates sprayed on the pedestrian zone in Berlin (Germany), to only name a few (see Figure \ref{pickpocket}). 

\begin{figure}[h!]
    \centering
\includegraphics[scale=0.5]{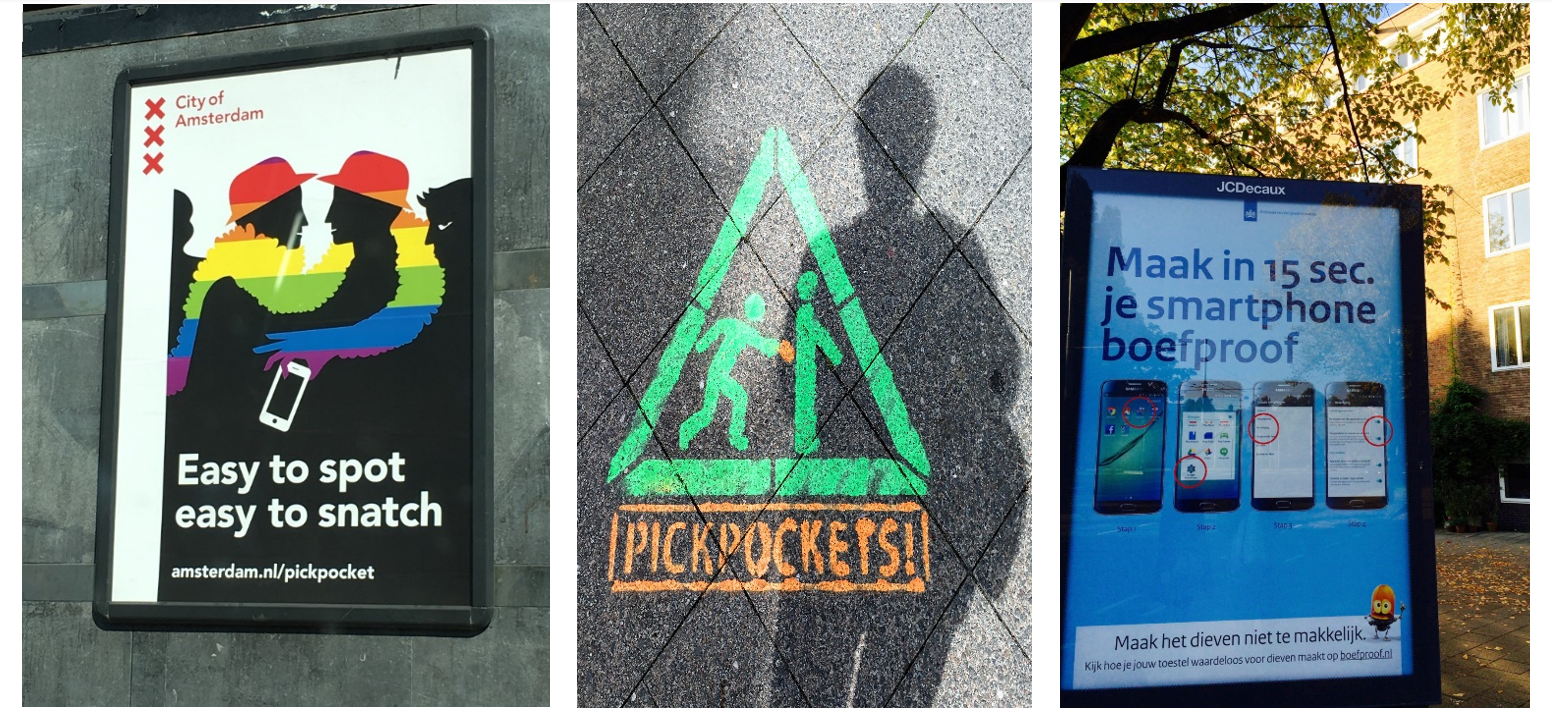}
    \caption{Examples of how to warn for pickpocketing (via billboards and sprayed templates).}
    \label{pickpocket}
\end{figure}

Next to these campaigns, governments also make use of more traditional security devices to fight pickpocketing. Common examples include the use of cameras surveillance systems and police patrolling in crowded areas such as metro stations and shopping halls. 

Any of these actions is carried out for the purpose of protecting European citizens by decreasing the number of pickpocket attempts. Naturally, these actions also require financial resources, and a delicate question for policy makers is therefore how to adequately allocate a finite budget among different protective resources (e.g., awareness campaigns and police patrolling) and potential pickpocket locations, in order to minimize the number of pickpocket attempts. Pickpocket locations are context dependent, but they typically include popular and crowded areas such as metro stations, market squares and sightseeing locations.

Inspired by this governmental budget allocation decision, we study a setting in which a police maker needs to allocate a financial budget over a number of protective resources and potential pickpocket locations, in order to minimize, per unit of time, the probability of pickpocketing over all potential pickpocket locations. Based on the actions developed within Europe, we consider that the protective resources can be categorized into two main categories: (i) the central ones that protect all the potential pickpocket locations (e.g., awareness campaigns via social media or the Boefproof app), and (ii) the local ones that protect a specific potential pickpocket location (e.g., assigning security cameras, billboards or police officers at a popular metro station). An important novelty in our model is the inclusion of the thieves' preferences towards potential pickpocket locations. We do so by using the multinomial logit model (MNL), a well-established method to describe the choice behavior of individuals in various domains (\citet{feng2022consumer}), including pickpocket location choice (see e.g., \citet{kuralarasan2022location}). Integrating a pickpocket location choice model in our budget allocation problem allows us to better capture the causal relationship between the budget allocated and the probability of pickpocketing at a specific location. Under the choice model, the thieves' preferences towards a specific location are characterized by a set of utility functions whose variables include the protective resources. The intuitive idea is that the more a potential pickpocket location is protected, the least it is attractive from the thieves' perspective (\citet{villegas2022spatial}). Regarding utility specification, we assume a log-linear utility function. Such a function ensures diminishing marginal utility, which is in line with our behavioral expectations that as the budget allocated increases, any additional dollar is contributing less to the utility. Moreover, non-linear specifications are often preferred to linear ones as they usually outperform them with regard to goodness-of-fit (\citet{sennhauser2010linear}). Since we integrate a pickpocket location \emph{choice} model into our budget allocation setting, we refer to our model as the \emph{choice-based resource allocation model}.

In this paper, we are interested in how the policy maker should optimally allocate the budget over the protective resources and potential pickpocket locations. Identifying such an optimal budget allocation is, however, complicated since our underlying optimization problem is non-convex. We therefore introduce a closely related convex optimization problem and prove that it has the same optimal solution as our original optimization problem. Then, by using results from convex optimization theory, we identify the optimal solution for this convex optimization problem and so also for our original optimization problem. Based on this result, we can conclude that the policy maker should first allocate the full budget over the central and local protective resources, proportionally to the sensitivity of the thieves' utility towards the protective resources. Subsequently, the policy maker should allocate the budget per local protective resource over the locations in a probabilistic way. The latter one can be described by a basic logit model, which is based on the initial attractiveness of the locations as specified by the thieves' utility. Using touristic attractions in Paris, we conclude with numerical experiments to show how our optimal budget allocation performs against various others budget allocations, proposed by policy makers from the field. From these experiments, we learn that optimally allocating the budget is especially important for heterogeneous locations, and that the amount of budget needed to keep the total pickpocket probability constant increases exponentially with the attractiveness of all locations.

The rest of the paper is organized as follows. Section \ref{LR} provides an overview of the main advancements in the research disciplines related to this paper: resource allocation for security and prevention, crime location decision-making, and choice-based optimization. In Section \ref{model}, we present the two decision-making processes involved in our study: the allocation of protective resources by a policy maker and the pickpocket location choice made by thieves. A first analysis of our choice-based resource allocation model is performed in Section \ref{newproblem} while optimal allocations are derived in Section \ref{sol}. Numerical experiments are presented in Section \ref{results} and we conclude this paper with final remarks in Section \ref{conclu}. Complete proofs of lemmas and theorems are given in the main text as they constitute our main contributions.

\section{Literature review} \label{LR}

This work considers a budget allocation decision over a number of protective resources and potential pickpocket locations. As such, we contribute to the literature on resource allocation for security and prevention. We further discuss this in Section \ref{literature1}. Moreover, in Section \ref{literature2}, we review some literature on crime location decision-making and choice-based optimization, two fields connected to our choice-based resource allocation model.

\subsection{Resource allocation for security and prevention} \label{literature1}

In the last four decades, significant efforts have been devoted in the development and application of resource allocation models to assist governmental agencies, such as police or homeland security departments, in making security and prevention decisions.

Pioneer work for police departments is done by \citet{kolesar1975queuing} and \citet{chaiken1978patrol}. In the late 70's, the authors were asked by the New York Police Department (NYPD) to come up with patrol car schedules that improve the fit between patrol car availability and demands for service. The authors used queuing theory and integer linear programming to come up with such schedules. Extensions and variations of these models can be found in \citet{green1984feasibility}, \citet{green1984multiple}, \citet{schaack1989n}, and \citet{kolesar1998insights}. It is worth noting that the patrol car allocation model introduced by \citet{chaiken1978patrol} is currently being used by more than 40 police departments in the US (\citet{green2004anniversary}). Other important decisions that are investigated for police departments are the optimal dispatching of police cars (see, e.g., \citet{dunnett2019optimising}), the joint decision of dispatching and locating police cars (see, e.g., \citet{adler2014location}), and the optimal partition of a city centre into patrol sectors (see e.g., \citet{camacho2015multi} and \citet{d2002simulated}). Interesting is to see that several of these quantitative studies were done in collaboration with police forces. \citet{curtin2010determining}, for example, investigated how the police of Dallas TX should optimally partition its city centre in patrol sectors and \citet{camacho2015decision} designed a method to identify patrolling areas in close collaboration with the Spanish National Police Corps.

Work for homeland security departments expanded significantly after the 9/11 attacks. A central question in many of these works is how a limited budget should be allocated over a number of potential attack locations. One could for instance think of the protection of some important locations and buildings, such as shopping districts, metro stations, and universities. In order to capture the strategic behavior of both a homeland security department and an offender (e.g., a terrorist) most of these papers have a game theoretical flavor (\citet{ azaiez2007optimal, bier2007choosing, bier2008optimal, zhuang2007balancing, zhuang2010reasons, zhuang2011secrecy, hausken2008strategic, shan2013hybrid, guan2017modeling, baron2018game} and \citet{ musegaas2022stackelberg}). Typical aspects in these stackelberg/simultaneous games are the cost effectiveness of security investments, the defender's valuations of the various targets, and the degree of the defender's uncertainty about the attacker's target valuations.

On a smaller scale level, these game theoretical models also found applications into airport security (\citet{kiekintveld2009computing}). For example, various types of security resources such as police officers, canine units, or checkpoints need to be allocated over multiple targets at airports. It is interesting again to see that the model of \citet{jain2010software} is used at the LA International airport for more than 10 years now, and various others airports are benefiting from this developed model as well. Other applications for homeland security departments can be found in the allocation of federal air marshals over airplanes (\citet{tambe2014computational}), the division of screening capacity over incoming shipments at ports (\citet{bier2011analytical, bakir2011stackelberg}), the deployment of coast guards over harbors (\citet{shieh2012protect}) and the allocation of rangers in wildlife national parks to fight poaching (\citet{yang2014adaptive}).

Clearly, our work can be recognized as another application in the field of resource allocation for security and protection. In line with most recent papers on security, our problem captures the behaviors of both a policy maker and a thief in a resource allocation setting. We do so by using a model from discrete choice theory, namely the multinomial logit (MNL) model, which represents the pickpocket location \textit{choice} process of a thief. For that reason, we shortly introduce the field of crime location decision-making, as well as the main advancements in choice-based optimization in the following sections.

\subsection{Crime location decision-making and choice-based optimization} 
\label{literature2}

The question why offenders commit a crime at a specific location and not somewhere else, is a  classic one in criminology. There exist various approaches to address this question, including the offender-based approach, target-based approach, mobility approach and the discrete choice approach (\citet{bernasco2005residential}). Our work follows this last one, which can be considered as the most recent advancement in the field. This approach is used for the first time in \citet{bernasco2005residential}. In their paper a discrete choice model, namely the conditional logit model, is used to analyze residential burglar's target area choice in the city centre of The Hague. Their work was the starting point for many others to use the discrete choice approach as a method to predict the decision making process of offenders in various forms of crime such as burglary, street robbery and pickpocketing (see, e.g., \citet{weisburd2009units, ruiter2017crime, curtis2021importance} and \cite{kuralarasan2022location}). Typically, in these papers, the authors try to identify some factors that might effect the decision of an offender to commit crime and test its significance by using data from the field. In our paper, we also consider such factors, namely local and central protective measures such as awareness campaigns via social media and police patrol.


In a broader, more theoretical, context, our resource allocation problem falls within the scope of \textit{choice-based optimization} (CBO), a growing body of literature in the operations research community (\cite{paneque2021integrating,paneque2022lagrangian,ROEMER2022}). \color{black} CBO allows to model the interplay between multiple actors, such as the government and thieves, whose decisions are important in the optimization problem under consideration. Moreover, CBO problems are probabilistic in nature, which make them attractive to capture the \textit{uncertain}  behavior of actors (e.g., the pickpocket location choice of a thief in our work). 

Several optimization problems have been investigated within a CBO framework. The most studied ones certainly come from the revenue management (RM) community and aim at maximizing revenues (and sometimes customer satisfaction) by deciding about product assortment (see e.g., the recent work of \cite{rusmevichientong2014assortment,nip2021assortment}) and/or pricing (see e.g., \cite{dong2019pricing,li2022pricing} for recent contributions). Various studies have included choice models in facility location problems, including school location (see e.g., \cite{castillo2015school}), health care facility location (see \cite{haase2015insights}), retail facility location (see e.g., \cite{muller2014customer}), and Park and Ride facility location (see e.g., \cite{aros2013p}). Other contributions are also present in literature on traffic assignment (see e.g., \cite{qian2013hybrid}), routing (see e.g., \cite{yang2017approximate}), and supply chain (see e.g., \cite{ROEMER2022}). Clearly, our paper is thus also enrolled in this line of existing literature on CBO, but for a new optimization problem and domain: a resource allocation problem for security and prevention.

\section{The Choice-based resource allocation model}
\label{model}
This section explains the two decision-making processes involved in our study: the allocation of budget amongst various local and central protective resources and potential pickpocket locations by a policy maker and the pickpocket location choice made by a thief\footnote{The thief could also represent an organized group of people planning to pickpocket.}.

\subsection{The policy maker's perspective}
\label{subsect:policymaker}
The optimization problem considered in this study is the one of a policy maker (e.g., mayor of a city) who needs to protect a set $N \subseteq \mathbb{N}$ of potential pickpocket locations. For example, the policy maker of Paris wants to protect its touristic attractions, such as the Eiffel tower, the Louvre Museum, the Sacré-Coeur and the Quartier Montmartre, from pickpocketing. In order to do so, a maximal budget of $R \in \mathbb{R}_{++}$ can be spent on several protective resources and potential pickpocket locations. Inspired by the innovative protection measures developed in the European Union in the last years, we divide the protective resources into two main categories: (i) the set of local protective resources $L \subseteq \mathbb{N}$ and the set of central protective resources $ C \subseteq \mathbb{N}$. For example, the policy maker of Paris may consider an awareness campaign via social media (a central protective resource) and assigning billboards and police officers to potential pickpocket locations (two local protective resources). For notational convenience, we assume $N \cap (C \cup L) = \emptyset$ and $L \cap C = \emptyset$ \textcolor{black}{and to avoid trivial cases, $ C \cup L \not = \emptyset$}. Moreover, we denote by $x_{ij} \in \mathbb{R}_{++}$ the budget allocated to local protective resource of type $j \in L$ to potential pickpocket location $i \in N$, and by $x_j \in \mathbb{R}_{++}$ the budget allocated to central protective resource $j \in C$. In compact form, we have $x = (x_{ij})_{i \in N, j \in L}, (x_j)_{j \in C})$, which we refer to as a \emph{budget allocation}. The set of feasible budget allocations is given by $$\mathscr{X} = \left\{ x \in \mathbb{R}^{{\vert N \vert \cdot \vert L \vert + \vert C \vert} }_{++} \bigg\vert \sum_{i \in N}\sum_{j \in L} x_{ij} + \sum_{j \in C} x_j \leq R\right\}.$$

Please, note that ($i$) the policy maker is allowed to spend its budget partially ($ii$) each type of protective resource gets some budget allocated, which can be arbitrarily small. The policy maker's objective is to find a budget allocation $x \in \mathscr{X}$ that minimizes the overall pickpocket probability. This overall pickpocket probability is described by a function $\mathbb{P}: \mathscr{X} \to [0,1]$ and represents the chance of pickpocketing per unit of time over all potential pickpocket locations together. In our example, it could for instance represent the chance of pickpocketing at the Eiffel tower, the Louvre Museum, the Sacré-Coeur and Quatier Montmarte together, per hour. Formally, the policy maker wants to solve problem $\mathscr{P}$, which is given by
\begin{equation}\mathscr{P} = \min_{x \in \mathscr{X}} \mathbb{P}(x).
\label{problemdefender}
\end{equation}

In the upcoming section, we  discuss the characteristics of the overall pickpocket probability in more detail. In particular, we describe how it captures the thief's choice towards potential pickpocket location, which is based on the protective resources allocated.

\subsection{The thief's perspective}

We assume that the thief chooses a location according to the multinomial logit model (MNL), the most widely used choice models in practice (\citet{bierlaire2017introduction}). Logit models assume that a decision maker chooses what he or she prefers from a finite, non-empty, set of alternatives, denoted by $\mathscr{C} \not = \emptyset$. In our setting, the decision maker is the thief who needs to choose from the set of potential pickpocket locations (i.e., $N$) and the option to not choose any of them, which we denote by 0.\footnote{As mentioned by \cite{campbell2019including}, including such opt-out option is widespread practice and recommended in choice models.} This option may, for instance, represent the choice to pickpocket in another city/district or to do nothing and stay at home. In our setting, the choice set is thus $\mathscr{C} = N \cup \{0\}$. A fundamental assumption behind logit models is that the decision maker is a rational utility maximizer. That means, the decision maker goes for choice $i \in \mathscr{C}$ that maximizes his or her given utility function. The exact specification of this utility function is however unknown and therefore modeled as a continuous random variable. More specifically, for any choice $i  \in \mathscr{C} \backslash \{\emptyset\}$, the utility function contains a deterministic function of the attributes of the alternatives, and a continuous error term, capturing the specification and measurement errors, and reads

\begin{equation*}\label{Eq:utility0}
U_{i}(x) = V_{i}(x)+ \varepsilon_{i} \mbox{ for all } x \in \mathscr{X}.
\end{equation*}

with $U_0(x) = 0$ for all $x \in \mathscr{X}$.\footnote{We decided to normalize this utility to zero, but this is not restrictive.} In our case, the deterministic utility function associated with choice $i \in \mathscr{C} \backslash \{0\}$ is:

\begin{align*}\label{Eq:utility}
V_{i}(x) &= \alpha_i - \sum_{j \in L}  \beta_j \ln x_{ij} - \sum_{j \in C}\beta_j \ln x_j \mbox{ for all } x \in \mathscr{X}, 
\end{align*}

with $\alpha_i \in \mathbb{R}_{+}$ and $\beta_j \in \mathbb{R}_{++}$. Parameter $\alpha_i$ is an alternative-specific constant that can be recognized as the initial interest in potential pickpocket location $i$ and captures exogenous factors, such as the number of tourists per time unit,  the distance to the thief's home, or the number of escape possibilities. Parameter $\beta_j$ characterizes the preferences of the thief regarding protective resources $j \in C \cup L$. It can thus be interpreted as the sensitivity of the thief towards the risk perceived with protective resource $j$. Furthermore, functions "$-\ln x_{ij}$" and "$-\ln x_j$", reflect the idea that the thief is aware of the protection measures in place at the different locations and take them into account when deciding which location to commit pickpocketing. In particular, the more a potential pickpocket location is protected, the least it is attractive from the thief’s perspective (see, e.g., \citet{villegas2022spatial}). This relationship is modelled in a log-linear fashion, since the marginal utility is unlikely to be constant for these protective resources (i.e., adding an extra dollar on top of an 1 million investment is not contributing the same as the first dollar invested). Finally, $\varepsilon_{i}$ is a continuous error term, capturing the specification and measurement errors. In the context of pickpocket location choice  (see, e.g., \citet{kuralarasan2022location}), it is custom to assume that all error terms are independent, identically, and extreme value distributed with location parameter $0$ and scale parameter $1$ (i.e., $\varepsilon_{i} \sim$ EV(0,1)), and we do so as well. Under this assumption, the probability that the thief decides to pickpocket at pickpocket location $i$ is
\begin{equation*}\label{Eq:logitprobability}
 \mbox{Pr}(U_{i}(x) \geq U_{j}(x), \forall j \in \mathscr{C}) = \frac{e^{ V_{i}(x) }}{\sum_{j \in \mathscr{C}} e^{ V_{j}(x)}} \mbox{ for all } i \in \mathscr{C} \mbox{ and all } x \in \mathscr{X},
\end{equation*}

which is the well-known multinomial logit expression, a widely used choice model (see e.g., \citet{frith2019modelling}). Now, by exploiting the definition of $V_i(x)$ for all $x \in \mathscr{X}$, the probability that potential pickpocket location $i \in N$ is selected by the thief, given
$x \in \mathscr{X}$, equals

\begin{equation*} \mathbb{P}_i(x) = \frac{ \mbox{exp}\left\{\alpha_i - \sum_{j \in L} \beta_j \ln(x_{ij}) - \sum_{j \in C}\beta_j \ln(x_j)\right\}}{1 + \sum_{i \in N} \mbox{exp}\left\{\alpha_i - \sum_{j \in L} \beta_j \ln(x_{ij}) - \sum_{j \in C}\beta_j \ln(x_j)\right\}}.
 \end{equation*}

Consequently, the overall pickpocket probability can be obtained by summing the above pickpocket probabilities over all potential pickpocket locations. Formally, for a given $x \in \mathscr{X}$, this probability $\mathbb{P}(x)$ equals
\begin{equation} \begin{aligned} \label{equation:probability} \mathbb{P}(x) &= \sum_{i \in N} \mathbb{P}_i(x) = \frac{\sum_{i \in N} \mbox{exp}\left\{\alpha_i - \sum_{j \in L} \beta_j \ln(x_{ij}) - \sum_{j \in C}\beta_j \ln(x_j)\right\}}{1 + \sum_{i \in N} \mbox{exp}\left\{\alpha_i - \sum_{j \in L} \beta_j \ln(x_{ij}) - \sum_{j \in C}\beta_j \ln(x_j)\right\}}.
\end{aligned} \end{equation}

By exploiting the relationship of (\ref{equation:probability}), we can now describe the policy maker optimization problem in full, namely

\begin{equation} \label{problem:P}\mathscr{P} = \min_{x \in \mathscr{X}} \frac{\sum_{i \in N} \mbox{exp}\left\{\alpha_i - \sum_{j \in L} \beta_j \ln(x_{ij}) - \sum_{j \in C}\beta_j \ln(x_j)\right\}}{1 + \sum_{i \in N} \mbox{exp}\left\{\alpha_i - \sum_{j \in L} \beta_j \ln(x_{ij}) - \sum_{j \in C}\beta_j \ln(x_j)\right\}} \end{equation}

In the next section, we study $\mathscr{P}$ and some of its properties in more detail. 

\section{An analysis of optimization problem $\mathscr{P}$}
\label{newproblem}

As a first step, we slightly reformulate the overall pickpocket probability $\mathbb{P}$ in a form that is more suitable for our exact analysis later on.
\begin{lemma} \label{lemma:rewritingprob} For every budget allocation $x \in \mathscr{X}$, the overall pickpocket probability equals
$$ \mathbb{P}(x) = \left(\sum_{i \in N}   \frac{\emph{exp}\left\{a_i\right\}}{\prod\limits_{j \in L}x_{ij}^{\beta_j} \prod\limits_{j \in C} x_j^{\beta_j}} \right)\Bigg / \left(1+ \sum_{i \in N} \frac{\emph{exp}\left\{a_i\right\}}{\prod\limits_{j \in L}x_{ij}^{\beta_j} \prod\limits_{j \in C} x_j^{\beta_j}}\right). $$
\end{lemma}

\begin{proof}
Let $x \in \mathscr{X}$ and $i \in N$. Then
\begin{equation*} \mbox{exp} \left\{\alpha_i - \sum_{j \in L} \beta_j \ln x_{ij} -\sum_{j \in C} \beta_j \ln x_j \right\} = \mbox{exp} \left\{\alpha_i +  \ln \left( \frac{1}{\prod\limits_{j \in C} x_j^{\beta_j} \prod\limits_{j \in L}x_{ij}^{\beta_j}}\right) \right\}  = \frac{\mbox{exp}(\alpha_i)}{\prod\limits_{j \in C} x_j^{\beta_j} \prod\limits_{j \in L}x_{ij}^{\beta_j}}. \label{equation:reformulattingB(x)}\end{equation*}
In the first equality, we use the properties $b \ln{a} = \ln(a^b)$ for any $a \in \mathbb{R}_{++}$, $b \in \mathbb{R}$ and $\ln{a} + \ln{b} = \ln{ab}$ for any $a,b \in \mathbb{R}_{++}$, and in the second equality, we use the property $\mbox{exp}\left\{\ln(y)\right\} = y$ for all $y \in \mathbb{R}_+$. Please, note that we can use these properties since we have $x \in \mathbb{R}^{\vert N \vert \cdot  \vert L \vert + \vert C  \vert }_{++}$. By applying the above result to (\ref{problem:P}) for any $i \in N$, we obtain the result.\end{proof}

We now illustrate how Lemma \ref{lemma:rewritingprob} can be applied in a  fictitious example. Note, parameters are not selected to represent reality, but to keep calculations simple and easy to follow.

\begin{example}
Suppose the policy maker of Paris wants to protect the Louvre Museum and the Eiffel tower against pickpocketing. If we label the Louvre Museum by "1" and the Eiffel tower by "2", then $N = \{1,2\}$. A schematic representation of the situation is given in Figure \ref{fig:2cities}.
\begin{figure}[h!]
    \centering
\includegraphics[scale=0.2]{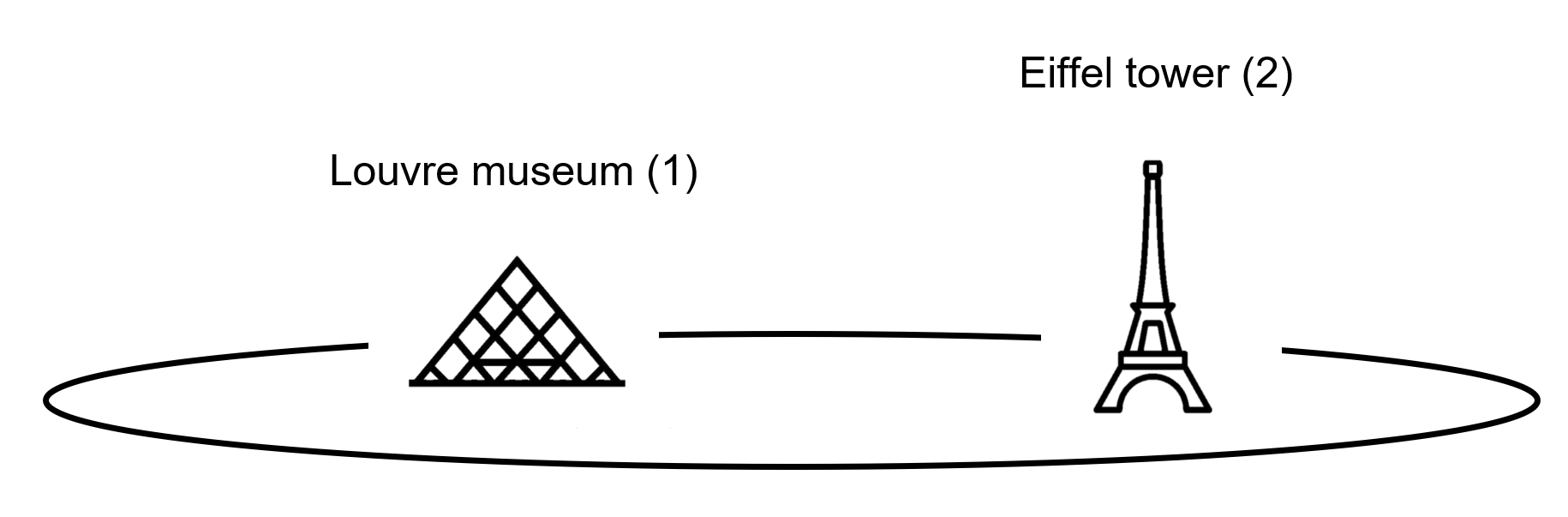}
    \caption{Schematic representation of Paris}
    \label{fig:2cities}
\end{figure}

The policy maker of Paris has a budget $R = 30$ to allocate over one central resource (digital campaign), denoted by $C = \{3\}$, and two local resources (cameras and billboards), denoted by $L = \{4,5\}$. Based on the number of pickpocket attempts, estimated per hour, and given the protective measures taken in previous years, the policy maker estimates the following utility parameters: $\alpha_1 = 6 \ln(3)$ and $\alpha_2 = 6 \ln(2)$, $\beta_3 = 1, \beta_4 = 3$ and $\beta_5 = 2$. The policy maker decides to split $R$ as follows: $x_3 = 15$, $x_{1,4} = 3$, $x_{2,4} =2$,  $x_{1,5} = 6$, and $x_{2,5} = 4$. Then, the overall pickpocket probability per hour equals
$$\mathbb{P}(x) = \left( \frac{729}{3^3 \cdot 6^2 \cdot 15^1} + \frac{64}{2^3 \cdot 4^2 \cdot 15^1}\right) \bigg / \left( \frac{729}{3^3 \cdot 6^2 \cdot 15^1} + \frac{64}{2^3 \cdot 4^2 \cdot 15^1} + 1 \right)  = \frac{1}{13},$$

Please, observe that $x$ is not an optimal budget allocation. For instance, for budget allocation $x' \in \mathscr{X}$ with $x'_3 =  10$, $x_{1,4}' = 3$, $x_{2,4}' = 2$, $x_{1,5}' = 9$, and $x_{2,5}' = 6$ we obtain $\mathscr{P}(x') = \frac{1}{19}$. Since $\mathbb{P}(x') = \frac{1}{19} < \frac{1}{13} = \mathbb{P}(x)$, budget allocation $x$ cannot be optimal. \label{example:example1}
\end{example}

In Example \ref{example:example1}, we identified two budget allocations which fully used the budget available. It is, however, not clear whether such type of budget allocation is optimal. As an example, consider a policy maker who spends its budget partially. Then, allocating the remaining budget to a specific potential pickpocket location lowers the attractiveness of that location, but also increases the \emph{relative} interest towards other locations. Ultimately, it is unclear how these two aspects balance out against each other, and so whether using the full budget available is optimal. We now discuss an example to get more grip on this.

\begin{example} \label{example:example2}
Consider a city centre with $N = \{1,2\}$, $R = 3$, $L = \{3\}$, $C = \emptyset$, $\alpha_1 = \alpha_2 = 0$, and $\beta_{1,3} = \beta_{2,3} = 4$. In Table \ref{tab:table2}, we represent the probabilities that the thief opts for location 1, location 2 or decides to not pickpocket, for two budget allocations. 
\begin{table}[h!]
    \centering
    \begin{tabular}{c|c|c|c} \hline
         $x = (x_{1,3},x_{2,3})$&  $\mathbb{P}_1(x)$ & $\mathbb{P}_2(x)$ & $1 - \mathbb{P}(x)$  \\ \hline \\
        $(1,1) $ & $\frac{1}{3}$ & $\frac{1}{3}$ & $\frac{1}{3}$ \\ \\ 
         $(2,1)$ & $\frac{1}{33}$ & $\frac{16}{33}$ & $\frac{16}{33}$ \\ \\
\hline
    \end{tabular}
    \caption{Pickpocket probability per location and probability of no crime for two budget allocations.}
    \label{tab:table2}
\end{table}

From Table $\ref{tab:table2}$, we learn that allocating one more unit of budget towards location 1 results in a drop in the pickpocket probability in location 1 of $\frac{10}{33}$. At the same time, the pickpocket probability in location 2 increases with $\frac{5}{33}$. This increase, however, is only half of the decrease realized in location 1. The other half ends up as an increased probability of no pickpocketing. So, in this example, allocating more budget decreases the overall pickpocket probability.  
\end{example}

We now show that allocating more budget always leads to a decrease in the overall pickpocket probability, and consequently, that it is optimal to spend the full budget.

\textcolor{black}{\begin{lemma} \label{lemma:fullR}
Every optimal budget allocation $x \in \mathscr{X}$ satisfies $\sum_{i \in N} \sum_{j \in L} x_{ij} + \sum_{j \in C} x_j = R$. That means, it is optimal to use the full budget $R$.
\end{lemma}}

\begin{proof}

First, we show that for any $A_1,A_2 \in \mathbb{R}_{++}$ with $A_1 < A_2$ and any $B \in \mathbb{R}_{+}$ we have

\begin{equation} \frac{A_1 + B}{A_1 + B + 1} < \frac{A_2 + B}{A_2 + B + 1}. \label{intermediateresult} \end{equation}

Let $A_1,A_2 \in \mathbb{R}_{++}$ with $A_1 < A_2$ and $B \in \mathbb{R}_{+}$. Then,
\begin{equation*} \begin{aligned} 0 > &\frac{A_1 - A_2}{(A_1 + B + 1)(A_2 + B + 1)} \\
&= \frac{(A_1 + B) (A_2 + B  + 1)}{(A_1 + B + 1)(A_2 + B + 1)} - \frac{(A_1 + B + 1) (A_2 + B)}{(A_1 + B + 1)(A_2 + B + 1)}\\
&= \frac{A_1 + B}{A_1 + B + 1} - \frac{A_2 + B}{A_2 + B + 1},
\end{aligned} \end{equation*}
where the inequality holds since $(A_1 + B + 1)(A_2 + B + 1) > 0$ and $A_1-A_2<0$.

Now, let $x \in \mathscr{X}$ be a budget allocation such that $\sum_{i \in N} \sum_{j \in L} x_{ij} + \sum_{j \in C} x_j = R + \varepsilon$ for some $\varepsilon > 0$. We will show that $x$ is not an optimal budget allocation by case condition. 

\begin{itemize}
    \item{
\emph{Case 1. $C \not = \emptyset$}

\medskip 

Assume, without loss of generality, that $1 \in C$. Moreover, let  $\overline{x}$ with $\overline{x}_{ij} = x_{ij}$ for all $i \in N$ and all $j \in L$, $\overline{x}_1 = x_1 + \epsilon$ and $\overline{x}_j = x_j$ for all $j \in C \backslash \{1\}$. Please, observe that $\overline{x} \in \mathscr{X}$, i.e., $\overline{x}$ is a feasible budget allocation. Now, let 
$$A_1 = \left(\sum_{i \in N}   \frac{\mbox{exp}\left\{a_i\right\}}{\prod\limits_{j \in L}\overline{x}_{ij}^{\beta_j} \prod\limits_{j \in C } \overline{x}_j^{\beta_j}} \right), A_2 = \left(\sum_{i \in N}   \frac{\mbox{exp}\left\{a_i\right\}}{\prod\limits_{j \in L}x_{ij}^{\beta_j} \prod\limits_{j \in C } x_j^{\beta_j}} \right) \mbox{ and } B = 0.$$

Note that $0 < A_1 < A_2$. Then, by (\ref{intermediateresult}) we obtain
$$ \mathbb{P}(\overline{x}) =  \frac{A_1 + B}{A_1 +  B + 1} < \frac{A_2 + B}{A_2 + B + 1} = \mathbb{P}(x).$$}
\item{\emph{Case 2. $C = \emptyset$}

\medskip

By the case condition, we have $L \not = \emptyset$. Assume, without loss of generality, that $1 \in N$ and $2 \in L$. Next, let $\overline{x}_{12} = x_{12} + \varepsilon$ and $\overline{x}_{ij} = x_{ij}$ for all $i \in N$ and all $j \in L$ with $(i,j)\not =(1,2)$. Please, observe that $\overline{x} \in \mathscr{X}$, i.e., $\overline{x}$ is a feasible budget allocation. Let 
$$A_1 =   \frac{\mbox{exp}\left\{a_1\right\}}{\prod\limits_{j \in L}\overline{x}_{1j}^{\beta_j} },  A_2 =    \frac{\mbox{exp}\left\{a_1\right\}}{\prod\limits_{j \in L}x_{1j}^{\beta_j}} \mbox{ and } B = \left(\sum_{i \in N \backslash \{1\}}   \frac{\mbox{exp}\left\{a_i\right\}}{\prod\limits_{j \in L}x_{ij}^{\beta_j}} \right).$$

Note that $0< A_1 < A_2$ and $B\geq0$. Then, by (\ref{intermediateresult}) we obtain
$$ \mathbb{P}(\overline{x}) =  \frac{A_1 + B}{A_1 +  B + 1} < \frac{A_2 + B}{A_2 + B + 1} = \mathbb{P}(x).$$}
\end{itemize}
In both cases, we have $\mathbb{P}(\overline{x}) < \mathbb{P}(x)$, which concludes this proof. \end{proof}

\begin{remark}
Although it is optimal to spend the full budget, it is not true that each budget allocation with fully spent budget dominates every budget allocation with partially spent budget. For instance, in Example \ref{example:example2} we  have $\mathbb{P}((2.5,0.5)) \approx 0.94 > \frac{2}{3} = \mathbb{P}((1,1))$. Basically, this implies that the policy maker should prefer more budget, but only if properly allocated. 
\end{remark}

We would like to finish this section by stating that our optimization problem $\mathscr{P}$ has a non-convex objective function. This may complicate the identification and characterisation of an optimal budget allocation. We illustrate the non-convexity by means of an example. 

\begin{example} \label{example:example3} Reconsider the setting of Example \ref{example:example2}. Due to the full budget result of Lemma \ref{lemma:fullR}, we know that $x_{1,3} + x_{2,3} = 3$. Consequently, we can describe the overall pickpocket probability in terms of $x_{1,3} \in (0,3)$. The result hereof is depicted in Figure \ref{figure:convexfunction1} on the next page.

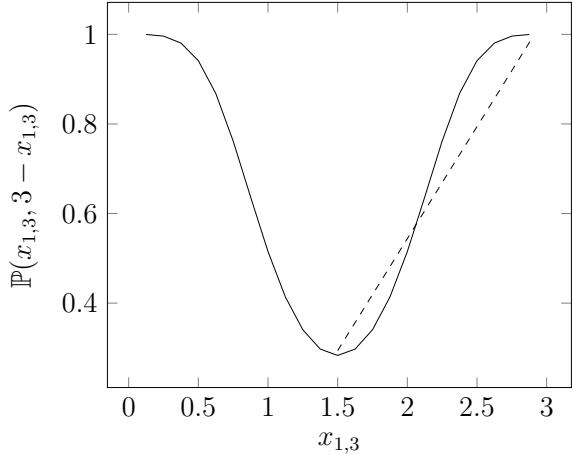
\begin{figure}[h!]
\centering
\begin{tikzpicture}[domain=0:3, , scale = 0.9]  \begin{axis}[     xlabel=$x_{1,3}$,    ylabel={$\mathbb{P}(x_{1,3}, 3 - x_{1,3})$}, legend style={at={(0.65,0.1)},anchor=west}] ]    \addplot[mark=none] {(1/x^4 + 1/(3-x)^4)/(1/x^4 + 1/(3-x)^4+ 1)};  
\addplot[domain=1.5:2.9, mark=none, dashed] {(-0.456 + 0.5*x)};
\end{axis}
\end{tikzpicture}
\caption{Example of non-convex objective function} \label{figure:convexfunction1}
\end{figure}

\newpage 
From Figure \ref{figure:convexfunction1}, we learn that the objective function is non-convex (see dashed line).

\end{example}

\section{Solving optimization problem $\mathscr{P}$}
\label{sol}

In this section, we focus on solving our optimization problem $\mathscr{P}$. That means, we are interested in finding characteristics of an optimal budget allocation. In doing so, we first introduce another, but strongly related, optimization problem, which we denote by

\begin{equation} \mathscr{B} = \min_{x \in \mathscr{X}} \mathbb{B}(x), \end{equation}

with $\mathbb{B}: \mathscr{X} \to \mathbb{R}$ be the objective function, which is defined as follows

\begin{equation} \mathbb{B}(x) = \sum_{i \in N}   \frac{\mbox{exp}\left\{a_i\right\}}{\prod\limits_{j \in C} x_j^{\beta_j} \prod\limits_{j \in L}x_{ij}^{\beta_j}} \mbox{ for all } x \in \mathscr{X}.\end{equation}

It is not hard to observe that optimization problem $\mathscr{P}$ and $\mathscr{B}$ have the same optimal solution(s). Objective function $\mathbb{P}$ can be recognized as objective function $\mathbb{B}$, divided by that same function plus one. This fraction is non-decreasing in its argument, and so if $\mathscr{B}$ is minimized, \textbf{$\mathscr{P}$} is as well (and vice versa). We now formalize this.

\begin{lemma} \label{lemma:similarproblem} Optimization problem $\mathscr{P}$ and $\mathscr{B}$ have the same optimal solution(s). \end{lemma}

\begin{proof} First, observe that optimization problem $\mathscr{P}$ has the same optimal solution as optimization problem $\mathscr{P}_1$, which is defined by      

$$ \mathscr{P}_1 = \max_{x \in \mathscr{X}} \left\{1 - \mathbb{P}(x)\right\}.$$

This holds, since ($i$) minimizing a certain objective function is the same as maximizing that objective function times minus one and ($ii$) adding a constant to the objective function does not effect the optimal solution. The objective function of
$\mathscr{P}_1$ can be rewritten as  $$1 - \mathbb{P}(x) = \frac{1}{1 + \sum_{i \in N}   \frac{\mbox{exp}\left\{a_i\right\}}{\prod\limits_{j \in C} x_j^{\beta_j} \prod\limits_{j \in L}x_{ij}^{\beta_j}}} \mbox{ for all } x\in \mathscr{X}.$$ Since $1 + \sum_{i \in N}   \frac{\exp\left\{a_i\right\}}{\prod\limits_{j \in C} x_j^{\beta_j} \prod\limits_{j \in L}x_{ij}^{\beta_j}} >0$ for all $x \in \mathscr{X}$, the above fraction is maximized by minimizing its denominator. Moreover, since adding a constant (i.e., +1) does not affect the optimal solution, the above fraction is maximized by minimizing $ \sum_{i \in N}   \frac{\mbox{exp}\left\{a_i\right\}}{\prod\limits_{j \in C} x_j^{\beta_j} \prod\limits_{j \in L}x_{ij}^{\beta_j}} = \mathbb{B}(x)$. Consequently, optimization problem $\mathscr{P}$ and $\mathscr{B}$ must have the same optimal solution(s).  \end{proof}

We also illustrate the result of Lemma \ref{lemma:similarproblem} by means of an example.

\begin{example}
Reconsider the setting of Example 2. In Figure \ref{figure:convexfunction5}, we plot $\mathbb{B}(x)$ and $\mathbb{P}(x)$ for all $x_{1,3} \in (0,3)$. Please, observe in Figure \ref{figure:convexfunction5} that the values of $\mathbb{P}(x)$ and $\mathbb{B}(x)$ do not coincide, but their optimal solution does. In particular, for both problems we have $x_{1,3}^* = 1.5$.

\begin{figure}[h!]
\centering
\begin{tikzpicture}[domain=0:3, scale = 0.9]  \begin{axis}[xmin=0, xmax=3, ymin=0,ymax=3,     xlabel=$x_{1,3}$, legend style={at={(0.72,0.12)},anchor=west}] ]    \addlegendentry{$\mathbb{P}(x)$} \addplot[mark=none] {(1/x^4 + 1/(3-x)^4)/(1/x^4 + 1/(3-x)^4 + 1)};
\addlegendentry{$\mathbb{B}(x)$} \addplot[domain=0.7:2.3, mark=none, dashed] {(1/x^4 + 1/(3-x)^4)};
\end{axis}\end{tikzpicture}\caption{$\mathbb{P}(x)$ and $\mathbb{B}(x)$ for various values of $x_{1,3}$.} \label{figure:convexfunction5}
\end{figure}
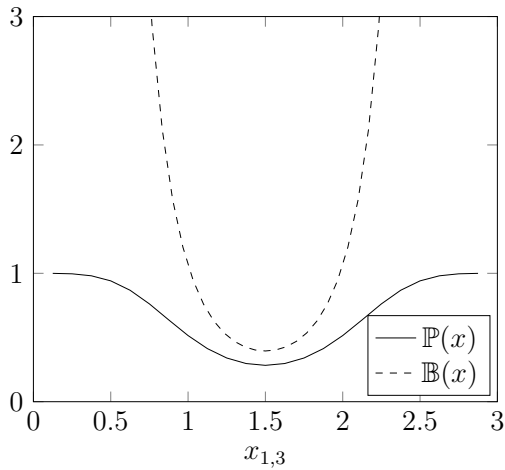
\end{example}

We now further investigate optimization problem $\mathscr{B}$ and its corresponding objective function. It turns out that $\mathbb{B}$ has a very desirable property, namely that it is convex. In order to show this, we first state some classical results about convexity for functions.

\begin{lemma} \label{lemma:sumconvex}
Let $X \subseteq \mathbb{R}^n$, $f,g: X \to \mathbb{R}$ be two convex functions and $\alpha_1,\alpha_2 \in \mathbb{R}_+$. Then, $\alpha_1 \cdot f+ \alpha_2 \cdot g$ is a convex function on $X$.
\end{lemma}
\begin{proof}
See \citet{bazaraa2013nonlinear}, chapter 3, page 99.
\end{proof}
\begin{lemma} \label{lemma:compositeconvex}
Let $g: \mathbb{R} \to \mathbb{R}$ be a non-decreasing convex function, $X \subseteq \mathbb{R}^n$, and $h: X \to \mathbb{R}$ be a convex function. Then, the composite function $f: X \to \mathbb{R}$ defined as  $g \circ h$ is convex on $X$. \end{lemma}
\begin{proof}
See \citet{bazaraa2013nonlinear}, chapter 3, page 99.
\end{proof}


\begin{lemma} \label{lemma:convexofB}
$\mathbb{B}$ is convex on $\mathscr{X}.$\end{lemma}

\begin{proof}

Since $\beta_j > 0$ for all $j \in C \cup L$, we can recognize $ - \beta_j \ln x_j$ for all $j \in C$ and $ - \beta_j \ln x_{ij}$ for all $j \in L$ and all $i \in N$ as convex functions on $\mathscr{X}$. Moreover, constant $\alpha_i$ for all $i\in N$ is also convex on $\mathscr{X}$. By Lemma \ref{lemma:sumconvex}, we know that the sum of convex functions is still convex on the same domain, and so for every $i \in N$,

$$\alpha_i -\sum_{j \in C} \beta_j \ln x_j - \sum_{j \in L} \beta_j \ln x_{ij}$$

is convex on $\mathscr{X}$. \\

Next, observe that the exponential function is a non-decreasing convex function on $\mathbb{R}$. Consequently, by Lemma \ref{lemma:compositeconvex}, we know that the composite function 
$$ \mbox{exp} \left\{\alpha_i-\sum_{j \in C} \beta_j \ln x_j - \sum_{j \in L} \beta_j \ln x_{ij}\right\}$$ 

is convex on $\mathscr{X}$ for all $i \in N$. Consequently, by Lemma \ref{lemma:sumconvex}, we also know that

$$\sum_{i \in N} \mbox{exp} \left\{\alpha_i-\sum_{j \in C} \beta_j \ln x_j - \sum_{j \in L} \beta_j \ln x_{ij}\right\}$$ 

is convex on $\mathscr{X}.$ By Lemma \ref{lemma:rewritingprob} (see the second line of the proof), we have
$$\sum_{i \in N} \mbox{exp} \left\{\alpha_i-\sum_{j \in C} \beta_j \ln x_j - \sum_{j \in L} \beta_j \ln x_{ij}\right\} = \sum_{i \in N}   \frac{\mbox{exp}\left\{a_i\right\}}{\prod\limits_{j \in C} x_j^{\beta_j} \prod\limits_{j \in L}x_{ij}^{\beta_j}} = B(x) \mbox{ for all } x \in \mathscr{X}.$$ 

Consequently, $\mathbb{B}$ is convex on $\mathscr{X}$, which concludes the proof. \end{proof}

In addition to the convexity of $\mathbb{B}$, we can also say something about our set of potential optimal budget allocations. It is formulated as all $x \in \mathbb{R}_{++}^n$ for which

\begin{equation} \label{linearequationR} \sum_{i \in N} \sum_{j \in L} x_{i,j} + \sum_{j \in C} x_j = R, \end{equation} 

where the equality results from Lemma \ref{lemma:fullR} directly. Note that $\mathbb{R}_{++}^n$ is an open convex set and (\ref{linearequationR}) a linear equality. Consequently, optimization problem $\mathscr{B}$ can be recognized as a problem with a convex objective function (Lemma \ref{lemma:convexofB}), a linear constraint ($\sum_{i \in N} \sum_{j \in L} x_{i,j} + \sum_{j \in C} x_j = R$) and an open feasible set ($\mathbb{R}_{++}^n$). Such a problem is often called a convex problem (see, e.g., \citet{freund2016optimality}), for which we present an interesting result below.

\begin{theorem} \label{theorem:convexproblem}
Suppose that $\mathscr{G}$ is a convex problem on $\mathscr{Y} \subseteq \mathbb{R}^n$, i.e., an optimization problem 
\begin{equation*} \begin{aligned}
\mathscr{G} := &\min f(y)\\
\mbox{s.t. } \hspace{4mm} &h(y) = 0 \\
&y \in \mathscr{Y},
\end{aligned} 
\end{equation*}
with $f$ convex on $\mathscr{Y}$,
$h$ linear on $\mathscr{Y}$, and $\mathscr{Y}$ an open convex set. If  $\overline{y} \in \mathscr{Y}$ is a feasible solution of $\mathscr{G}$, and  $\overline{y}$ together with multiplier $\lambda \in \mathbb{R}$ satisfies
$$\nabla f(\overline{y}) + \lambda \nabla h(\overline{y}) = (0,0,,\ldots,0)^T,$$

then $\overline{y}$ is the global optimal solution of $\mathscr{G}$ (on $\mathscr{Y}$).
\end{theorem}

\begin{proof}
See \citet{freund2016optimality}, page 20.
\end{proof}

Clearly, we can apply Theorem \ref{theorem:convexproblem} to our convex problem $\mathscr{B}$:  we can show the existence of a feasible solution $x^* \in \mathscr{X}$ that satisfies the criteria of Theorem \ref{theorem:convexproblem}, implying that $\mathscr{B}$ has a global optimal solution. This solution is the optimal budget allocation for $\mathscr{B}$ and $\mathscr{P}$.

\begin{theorem} \label{theorem:optimalallocation1}
The optimal budget allocation $x^*$ of $\mathscr{B}$, and also of $\mathscr{P}$, is given by
\begin{equation}
x_j^* = \frac{\beta_j}{\sum\limits_{i \in L \cup C} \beta_i} \cdot R  \emph{ for all } j \in C, 
\end{equation}
\begin{equation} \label{metbeta's}
    x_{ij}^* = \frac{\beta_j}{\sum\limits_{m \in L \cup C} \beta_m} \cdot R \cdot  \frac{\emph{exp}\left\{\frac{\alpha_i}{1 + \sum_{k \in L} \beta_k}\right\}}{ \sum_{m \in N} \emph{exp}\left\{\frac{\alpha_m}{1 + \sum_{k \in L} \beta_k}\right\}} \emph{ for all } i \in N \emph{ and all } j \in L. 
\end{equation}
\end{theorem}

\begin{proof}

Let 
\begin{equation*} \begin{aligned}
&x_j^* = \frac{\beta_j}{\sum\limits_{i \in L \cup C} \beta_i} \cdot R  \emph{ for all } j \in C, \\
    &x_{ij}^* = \frac{\beta_j}{\sum\limits_{m \in L \cup C} \beta_m} \cdot R \cdot  \frac{\exp\left\{\frac{\alpha_i}{1 + \sum_{k \in L} \beta_k}\right\}}{ \sum_{m \in N} \exp\left\{\frac{\alpha_m}{1 + \sum_{k \in L} \beta_k}\right\}} \emph{ for all } i \in N \emph{ and all } j \in L\\
    &\lambda = -\frac{\left( \sum_{i \in N} \exp \left\{  \frac{a_i}{1 + \sum_{k \in L} \beta_k}\right\}\right)^{1 + \sum_{k \in L} \beta_k}   \left( \sum_{k \in L \cup C} \beta_k\right)^{1 + \sum_{k \in L \cup C} \beta_k}}{R^{1 + \sum_{k \in L \cup C} \beta_k} \prod_{k \in L \cup C} \beta_k^{\beta_k}}.\end{aligned}
\end{equation*}

We now show that $x^*$ is feasible. Since $\beta_k >0$ for all $k \in L \cup C$ and $R > 0$, we have $x_j^* >0$ for all $j \in C$ and $x_{i,j}^* > 0$ for all $i \in N$ and all $j \in L$. Moreover,
$$ \sum_{j \in C} x_j^* + \sum_{i \in N} \sum_{j \in L} x_{i,j}^* = \frac{\sum_{j \in C} \beta_j}{\sum_{k \in C \cup L} \beta_k} \cdot R + \frac{\sum_{j \in L} \beta_j}{\sum_{k \in C \cup L} \beta_k} \cdot R = R.  $$
Hence, $x^* \in \mathscr{X}$, which implies that $x^*$ is feasible. It now remains to show that

$$\nabla \mathbb{P}(x^*) + \lambda \nabla \left(R - \sum_{j \in C} x_j^* - \sum_{i \in N} \sum_{j \in L} x_{ij}^*\right) = \left(\begin{matrix} 0 \\ \vdots \\ 0 \end{matrix} \right).$$

The gradient of $\mathbb{P}(x)$ for all $x \in \mathscr{X}$ is defined by
$$ \nabla \mathscr{P}(x) =  \left( \begin{matrix} \frac{d}{dx_{ij}}  \mathbb{P}(x) & \mbox{ for all } i \in N, j \in L \\
\frac{d}{x_j} \mathbb{P}(x) & \mbox{ for all } j \in C\end{matrix}\right) \mbox{ for all } x \in \mathscr{X}.$$

We now evaluate $x^*$ for the two types of partial derivatives. Let $i \in N$ and $j \in L$, then

\begin{equation*} \begin{aligned} \frac{d}{dx_{ij}}  \mathbb{P}(x^*) &= -\frac{\beta_j\exp\{a_i\}} {x_{i,j}^*\prod_{k \in C} {x_k^*}^{\beta_k} \prod_{k \in L} {x_{i,k}^*}^{\beta_k}} \\
&= -\frac{\left( \sum_{m \in N} \exp\left\{ \frac{a_m}{1 + \sum_{k \in L} \beta_k}\right\}\right)^{1 + \sum_{k \in L}\beta_k} \left(\sum_{k \in L \cup C} \beta_k\right)^{1 + \sum_{k \in L}\beta_k} } {R^{1 + \sum_{k \in L} \beta_k} \prod_{k \in C} {x_k^*}^{\beta_k}} \\
&= -\frac{\left( \sum_{m \in N} \exp \left\{  \frac{a_m}{1 + \sum_{k \in L} \beta_k}\right\}\right)^{1 + \sum_{k \in L} \beta_k}   \left( \sum_{k \in L \cup C} \beta_k\right)^{1 + \sum_{k \in L \cup C} \beta_k}}{R^{1 + \sum_{k \in L \cup C} \beta_k} \prod_{k \in L \cup C} \beta_k^{\beta_k}} \\
&= \lambda,\end{aligned} \end{equation*}

where we substituted $x_{ij}^*$ and $x_{i,k}^*$ for all $i \in N$ and all $k \in L \backslash \{j\}$ in the second equality and $x_k^*$ for all $k \in C$ in the third equality. Now, let $j \in C$. Then, 
\begin{equation*} \begin{aligned} \frac{d}{dx_{j}}  \mathbb{P}(x^*) &= \sum_{i \in N} \frac{ - \beta_j\exp\{a_i\}} {x_{j}^*\prod_{k \in C} {x_k^*}^{\beta_k} \prod_{k \in L} x_{i,k}^*} \\
& = -\frac{\sum_{k \in L \cup C} \beta_k}{R} \sum_{i \in N} \frac{\exp\{a_i\}} {\prod_{k \in C} {x_k^*}^{\beta_k} \prod_{k \in L} x_{i,k}^*} \\
& = -\frac{\left(\sum_{k \in L \cup C} \beta_k\right)^{1 + \sum_{k \in C} \beta_k}}{R^{1 + \sum_{k \in C} \beta_k} \prod_{k \in C} \beta_k^{\beta_k}} \sum_{i \in N} \frac{\exp\{a_i\}} {\prod_{k \in L} x_{i,k}^*} \\
& = -\frac{\left(\sum_{k \in L \cup C} \beta_k\right)^{1 + \sum_{k \in C \cup L} \beta_k}}{R^{1 + \sum_{k \in C \cup L} \beta_k} \prod_{k \in C \cup L} \beta_k^{\beta_k}}  \sum_{i \in N} \frac{\exp\{a_i\}} {\prod_{k \in L} \left(\frac{\exp\left\{\frac{\alpha_i}{1 + \sum_{m \in L} \beta_m}\right\}}{ \sum_{m \in N} \emph{exp}\left\{\frac{\alpha_m}{1 + \sum_{m \in L} \beta_m}\right\}}\right)^{\beta_k}} 
\end{aligned} \end{equation*}

\begin{equation*}\begin{aligned}
& = -\frac{\left(\sum_{k \in L \cup C} \beta_k\right)^{1 + \sum_{k \in C \cup L} \beta_k}}{R^{1 + \sum_{k \in C \cup L} \beta_k} \prod_{k \in C \cup L} \beta_k^{\beta_k}}  \sum_{i \in N} \frac{\exp\{a_i\}} { \frac{\exp\left\{\frac{\alpha_i (\sum_{m \in L} \beta_m)}{1 + \sum_{m \in L} \beta_m}\right\}}{ \left(\sum_{m \in N} \emph{exp}\left\{\frac{\alpha_m}{1 + \sum_{m \in L} \beta_m}\right\}\right)^{\sum_{m \in L} \beta_m}}} \\
&= -\frac{\left( \sum_{i \in N} \exp \left\{  \frac{a_i}{1 + \sum_{k \in L} \beta_k}\right\}\right)^{1 + \sum_{k \in L} \beta_k}   \left( \sum_{k \in L \cup C} \beta_k\right)^{1 + \sum_{k \in L \cup C} \beta_k}}{R^{1 + \sum_{k \in L \cup C} \beta_k} \prod_{k \in L \cup C} \beta_k^{\beta_k}} \\
&= \lambda,\end{aligned} \end{equation*}

where we substituted $x_j^*$ in the second equality, $x_k^*$ for all $k \in C \backslash \{j\}$ in the third equality, and $x_{i,k}^*$ for all $i \in N$ and $k \in L$ in the fourth equality. In the fifth equality we used $\exp(a) \cdot \exp(b) = \exp(a+b)$ for all $a,b \in \mathbb{R}$. In the sixth equality, we used that $\exp(a) / \exp(b) = \exp(a-b)$ for all $a,b \in \mathbb{R}$ and merged some expressions.

The gradient of $\left(R - \sum_{j \in C} x_j - \sum_{i \in N} \sum_{j \in L} x_{ij}\right)$ for all $x \in \mathscr{X}$ is defined by
\begin{equation*} \begin{aligned} &\nabla \left(R - \sum_{j \in C} x_j - \sum_{i \in N} \sum_{j \in L} x_{ij}\right) \\
= & \left( \begin{matrix} \frac{d}{dx_{ij}}  \left(R - \sum_{j \in C} x_j - \sum_{i \in N} \sum_{j \in L} x_{ij}\right)  & \mbox{ for all } i \in N, j \in L \\
\frac{d}{x_j} \left(R - \sum_{j \in C} x_j - \sum_{i \in N} \sum_{j \in L} x_{ij}\right)  & \mbox{ for all } j \in C\end{matrix}\right) \mbox{ for all } x \in \mathscr{X}. \end{aligned} \end{equation*}

We now evaluate $x^*$ for the two types of partial derivatives. Let $i \in N$ and $j \in L$, then
\begin{equation*} \begin{aligned} \frac{d}{dx_{i}}  \left(R - \sum_{k \in C} x_k^* - \sum_{k \in N} \sum_{m \in L} x_{km}^*\right) &= -1
\end{aligned} \end{equation*}
Now, let $j \in C$. Then, 
\begin{equation*} \begin{aligned} \frac{d}{dx_{j}}  \left(R - \sum_{k \in C} x_k^* - \sum_{k \in N} \sum_{m \in L} x_{km}^*\right) &= -1
\end{aligned} \end{equation*}

Consequently, for $x^*$ we obtain
$$\nabla \mathbb{P}(x^*) + \lambda \nabla \left(R - \sum_{j \in C} x_j^* - \sum_{i \in N} \sum_{j \in L} x_{ij}^*\right) = \left(\begin{matrix} \lambda \\ \vdots \\ \lambda \end{matrix} \right) + \lambda \left(\begin{matrix} -1 \\ \vdots \\ -1 \end{matrix} \right) = \left(\begin{matrix} 0 \\ \vdots \\ 0 \end{matrix} \right).$$

Hence, by Theorem \ref{theorem:convexproblem}, $x^*$ is an optimal budget allocation of $\mathscr{B}$. By Lemma \ref{lemma:similarproblem}, $x^*$ is also an optimal budget allocation of $\mathscr{P}$, which concludes this proof. \end{proof}

Theorem \ref{theorem:optimalallocation1} tells us that there is an intuitive way to optimally allocate the budget among the protective resources. The policy maker should first divide the full budget $R$ over the central and local protective resources, proportionally to their marginal effects, captured by the $(\beta_j)_{j \in C \cup L}$ parameters. Intuitively it means the budget is first allocated within the protective resources based on their protective impact, as perceived by the thief. This confirms the intuition that policy makers want to allocate more budget to protective resources for which thieves are more risk-sensitive. Interestingly, the budget for local protective resources is then further distributed among the different potential pickpocket locations according to a logit model where utilities are based on the parameters only. Intuitively this result seems to indicate that more budget is allocated towards locations that have a higher intrinsic attractiveness for pickpocketing, captured by the $(\alpha)_{i \in N}$, regardless of the protective resources.

We conclude this section with an example illustrating the result of Theorem \ref{theorem:optimalallocation1}.

\begin{example}
Reconsider the setting of Example \ref{example:example1}. Due to Theorem \ref{theorem:optimalallocation1}, the optimal budget allocation is $x_{3}^* = 5$, $x_{1,4}^* = 9$, $x_{2,4}^* = 6$, $x_{1,5}^* = 6$, and $x_{2,5}^* = 4$ with $\mathbb{P}(x^*)  = \frac{5}{545} < \frac{1}{13} =\mathbb{P}(x)$.
\end{example}

\section{Numerical experiments: the pickpocket plan in Paris}
\label{results}

With a record number of more than 40 million visitors registered in 2018, Paris remains one of the most visited city in the world (\citet{sortira}). With no surprise pickpocketing is thus a major problem in the French capital. This was also confirmed by a news bulletin in June 2022 (see \citet{telegraaf}), stating that the police of Paris had reactivated its ‘pickpocket plan’ requesting police to patrol touristic attractions such as the Champs-Élysées, Montmartre, the Eiffel Tower, the Louvre museum and the Seine quays. 

Although we do not have access to the details of the pickpocket plan, we can imagine that policy makers of Paris allocate police patrols with simple intuitive rules\footnote{ \citet{ludwig2017measuring} studied the resource allocation processes of the UK’s territorial police forces. In particular, they investigated 43 Policing areas in England and Wales together and found out that often simple resource allocation rules, such as proportional (to demand) or static/non-flexible, are applied. We can imagine that other national police forces apply similar types of resource allocation rules.}. For example, they could decide to assign the same number of patrols over the potential pickpocket locations, or decide to assign patrols proportionally to the number of tourists per potential pickpocket location. Next to police patrol allocation, decisions concerning other types of protective resources, either central or local, could have been taken similarly. Motivated by this, we further investigate the Paris example, as presented in Section \ref{newproblem}, to analyze the performance of these simple intuitive rules. In particular, we consider the following two rules

\begin{enumerate}
    \item \textbf{$\text{Central and Location-Equal (CLE)}$ rule} first allocates a fraction $\gamma \in (0,1)$ of the budget $R$ to the central protective resources and the remaining fraction ($1-\gamma$) of $R$ to the local protective resources. The central budget is then allocated equally over the number of central protective resources and the local budget equally over the number of local protective resources and potential pickpocket locations. Hence, each potential pickpocket location get allocated the same amount of local protective resources. Formally, for any $\gamma \in (0,1)$ the rule reads 

\begin{equation} \begin{aligned}
    x^{\text{CLE}}_j &= \gamma \cdot \frac{R}{\vert C \vert}, \text{ for all } j \in C \\
    x^{\text{CLE}}_{ij} &= (1-\gamma) \cdot \frac{R}{ \vert L \vert \times \vert N \vert}, \text{ for all } i \in N \text{ and all } j \in L.
    \end{aligned}
\end{equation}

    \item \textbf{$\text{Central-Equal and Location-Proportional (CELP) rule}$} coincides with the CLE-rule except that the local budget is now allocated proportionally to the initial attractiveness of the potential pickpocket locations (i.e., proportionally to $(\alpha_i)_{i \in N}$). Formally, for any $\gamma \in (0,1)$, the rule reads

\begin{equation} \begin{aligned}
    x^{\text{CELP}}_j &= \gamma \cdot \frac{R}{\vert C \vert}, \text{ for all } j \in C \\ \\
    x^{\text{CELP}}_{ij} &= (1-\gamma)  \cdot \frac{R}{\vert L \vert} \cdot \frac{\alpha_i}{\sum_{k \in L} \alpha_k}, \text{ for all } i \in N \text{ and all } j \in L.
    \end{aligned}
\end{equation}
\end{enumerate}

We create a set of numerical experiments around our Paris example to analyze how these two rules perform in terms of overall pickpocket probability. In Table \ref{allocsexample1} we present the first results: we depict the optimal budget allocation (represented in bold on the first row), and some $\text{CLE}$ and $\text{CELP}$ allocations for values $\gamma \in \{0.25, 0.50, 0.75\}$. \\

\begin{table}[h!]
    \centering
    \begin{tabular}{|c|ccccc|cc|c|}
            \hline
      Allocation  &    \multicolumn{5}{c}{Allocations} & 
        \multicolumn{2}{c}{Pickp. probability} & Overall pickp. \\
          rule  & $x_3$ & $x_{14}$ & $x_{24}$ & $x_{15}$  & $x_{25}$  &  $N = \{1\}$ &  $N = \{2\}$ &  probability \\
\cline{1-9}
         OPTIMAL  & \textbf{5.00} & \textbf{9.00} & \textbf{6.00} & \textbf{6.00} & \textbf{4.00} & \textbf{0.55\%} & \textbf{0.37\%}  &  \textbf{0.92\%} \\
\cdashline{1-9}
         $\text{CLE}(\gamma = 0.25)$   & 7.50 & 5.625  & 5.625 & 5.625  & 5.625 &  1.69\% & 0.15\%  & 1.84\%  \\
         $\text{CLE}(\gamma = 0.50)$   & 15.00 & 3.75 & 3.75 & 3.75 & 3.75 & 6.12\% & 0.54\%  &  6.65\% \\
         $\text{CLE}(\gamma =  0.75)$   & 22.50 & 1.875 & 1.875 & 1.875 & 1.875  & 55.46\% & 4.87\%  & 60.33\% \\
\cdashline{1-9}
     $\text{CELP}(\gamma =  0.25)$   & 7.50 & 6.898 &	4.352 &	6.898 &	4.352 & 0.62\% &	0.54\% &	1.16\% \\
         $\text{CELP}(\gamma = 0.50)$   & 15.00 & 4.599 &	2.901 &	4.599 &	2.901 & 2.26\% &	1.99\% &	4.25\%
    \\
         $\text{CELP}(\gamma = 0.75)$   & 22.50 & 2.299 &	1.451 &	2.299 &	1.451 & 25.90\% &	22.74\% &	48.64\% \\
        \hline
    \end{tabular}
    \caption{Allocations and pickpocket probabilities stemming from CLE, CELP and the optimal budget allocation.}
    \label{allocsexample1}
\end{table}

From Table \ref{allocsexample1}, we can make two observations. First of all, we see that the overall pickpocket probability is strictly lower under the CELP rule. This is something we also see in different numerical experiments not reported here (i.e., for different $\gamma$'s). Hence, it is important to keep the initial attractiveness into account when allocating the budget. Secondly, we learn that the overall pickpocket probability can differ significantly for different values of $\gamma$. For instance, allocating 25\% of the total budget to central resources leads to an absolute increase of 59.41\% for CLE and 48.72\% for CELP, with respect to the minimal overall pickpocket probability. At the same time, we see that allocating 25\% of the total budget to central resources leads to an increase in overall pickpocket probability of 0.24\% and 0.9\%, respectively. These numbers are very close to what is maximally achievable in terms of pickpocket probability reduction. In summary, in this example, it is possible to come close to the minimal overall pickpocket probability with both intuitive rules, and in particular with the CELP rule. However, for both intuitive rules, it remains a delicate task how to allocate the total budget over the central and local resources. If the budget is not allocated properly over the central and local resources (i.e., $\gamma$ is not selected properly), the performance of simple intuitive allocation rules can be poor. 


In Figure \ref{newnew}, we illustrate the performance of the best possible CLE and CELP allocations against the optimal budget allocation for $(a_1,a_2) \in \{(1,9), (1.5,8.5),...,(9,1)\}$. All other parameters are kept constant and coincide with those set in the Paris example. The best possible CLE and CELP allocations are calculated by selecting the CLE and CELP allocation for which the overall pickpocket probability is the lowest amongst  $\gamma \in \{0.01, 0.02,...,0.99\}$.

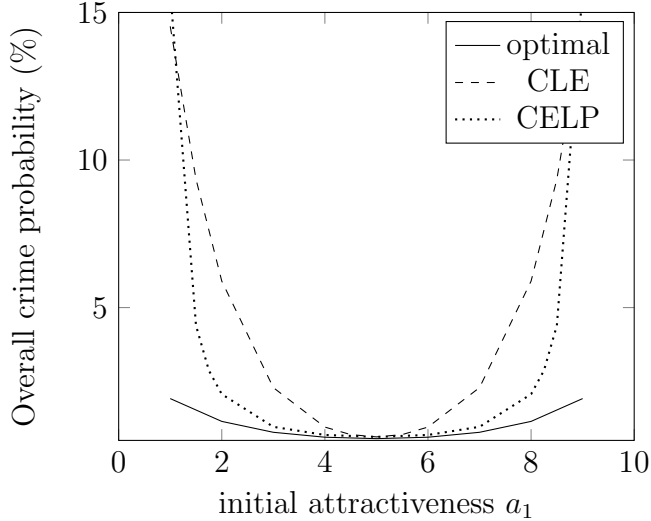
\begin{figure}[h!]
\centering
\begin{tikzpicture}[domain=1:9, scale = 1]  \begin{axis}[xmin=0, xmax=10, ymin=0.5,ymax=15,     xlabel= initial attractiveness $a_1$,, ylabel = Overall crime probability $(\%)$
]   
\addlegendentry{optimal} \addplot[black] coordinates {
    (1,1.917129824
) (2, 1.143937548
) (3, 0.774803214) (4, 0.607834909
) (5,0.559720147
) (6, 0.607834909) (7, 0.774803214) (8, 1.143937548) (9, 1.917129824)};
\addlegendentry{CLE} \addplot[dashed] coordinates {
    (1,14.52925656
) (1.5, 9.351659936
) (2, 5.897425276
) (3, 2.288337607
) (4, 0.9514133
) (4.5, 0.697042988
) (5,0.618638942
) (5.5, 0.697042988)
(6, 0.9514133
) (7, 2.288337607) (8, 5.897425276) (8.5, 9.351659936
)(9, 14.52925656)};
    \addlegendentry{CELP} \addplot[dotted, thick] coordinates {
    (1,15.76383679
) (1.5, 4.391975834) (1.75, 2.869354129)
(2, 2.065890931
) (3, 0.960005599
) (4, 0.684716592
) (5,0.618638942
) (6, 0.684716592) (7, 0.960005599) (8, 2.065890931) (8.25, 2.869354129) (8.5, 4.391975834) (9, 15.76383679)};
\end{axis}\end{tikzpicture}\caption{Overall pickpocket probability for optimal budget, CLE and CELP allocations with $a_1 \in [1,9]$.} \label{newnew}
\end{figure}

From Figure \ref{newnew}, we can make two observations. First of all, we see that both CLE and CELP perform quite well when the potential pickpocket locations are similar in terms of initial attractiveness. In particular, if both locations are equally attractive (i.e., $a_1 = a_2 =5$), both CLE and CELP show only a 0.05\% gap from the optimal overall pickpocket probability. However, if the attractiveness of the locations starts to differ, the performance of CLE and CELP decreases rapidly. For instance, if the second location is 9 times as attractive as the first location (i.e., $ a_1 =1, a_2=9)$, the total pickpocket probability of CLE and CELP is almost 14 times as high as under the optimal budget allocation. Secondly, we observe that CELP does no longer dominate CLE for all instances. For instance, if $a_1=1, a_2=9$, the CLE rule performs slightly better. However, in most cases, CELP rule is still preferable. Notably, the optimal $\gamma$ is the value closest to $\frac{\beta_3}{\beta_3 + \beta_4+\beta_5}$ for both CLE and CELP. This is in line with the central versus location resource division under the optimal budget allocation.

In summary, this experiment generalizes our previous observations that simple intuitive rules can perform well under certain conditions. In particular, this is the case if the initial attractiveness of pickpocket locations is roughly the same and the budget is allocated properly over the central and local protective resources. In other cases, however, when the initial attractiveness differs more, but also when the total budget is not allocated properly over the central and local protective resources, the performance of simple rules can be poor.  \\

Next to the study of simple and intuitive rules, we also investigate how the policy makers of the city of Paris should deal with their budget during special events or days, such as 'Le Quatorze Juillet' or a world cup football final. These days and events may be of particular interest for thieves, due to the extra number of tourists. Consequently, it is interesting for the city of Paris to identify how to (differently) allocate their budget during such days and events, and what is needed in terms of extra budget in order to keep the overall pickpocket probability stable. For the first question, we investigate what happens with the optimal allocation of the budget when the initial attractiveness of all locations increases. We do so by studying our Paris example for $(a_1,a_2) \in \{(6k\ln(3), 6k\ln(2)\}_{k=1,1.1,..,1.4}$. In Table \ref{allocsexample5}, we present the corresponding optimal budget allocations and overall pickpocket probabilities. 
\begin{table}[h!]
    \centering
    \begin{tabular}{c|ccccc|c}
            \hline
      Initial  &    \multicolumn{5}{c}{Allocations}  & Overall crime \\
          attractiveness  & $x_3$ & $x_{14}$ & $x_{24}$ & $x_{15}$  & $x_{25}$  &  probability \\
\cline{1-7}
         $(6.0 \ln(3), 6.0 \ln(2))$  & \textbf{5.000} & \textbf{9.000} & \textbf{6.000} & \textbf{6.000} & \textbf{4.000} &   \textbf{0.92\%} \\
\cdashline{1-7}
         $(6.6 \ln(3), 6.6 \ln(2))$   & 5.000 & 9.145  & 5.855 & 6.097  & 3.093  & 1.60\%  \\
         $(7.2 \ln(3), 7.2 \ln(2))$   & 5.000 & 9.289 & 5.711 & 6.193 & 3.807  & 2.78\% \\
         $(7.8 \ln(3), 7.8 \ln(2))$   & 5.000 & 9.432 & 5.568 & 6.288 & 3.712  & 4.81\% \\
   \hline
    \end{tabular}
    \caption{Allocations and crime probabilities stemming for various combinations of initial attractiveness.}
    \label{allocsexample5}
\end{table}

From Table \ref{allocsexample5}, we see that the total amount of resources allocated to central and local remains stable for various combinations of the initial attractiveness (i.e., remains 5 for central and 25 for local, each time). This also holds true for the total amount of resources allocated per location (i.e., $x_{14} + x_{24} = 15$ and $x_{15} + x_{25} = 10$). Basically, this tells us that the division of the total budget amongst central and local resources, as well as the amount allocated per location, is independent of the initial attractiveness of the locations. This is not a coincidence. It results from the proportional structure of the optimal allocation rule.

Secondly, we observe that the division of the budget per potential pickpocket location and resource type combination is not constant (e.g., $9.00 \not= 9.145$ for potential pickpocket location 1 and resource type 4).  In particular, it deviates such that ratio $x_{14}/x_{15} = x_{24}/x_{25} = \beta_4/\beta_5$ remains stable. In other words, the resources should be treated in the same proportional way amongst the locations. Since some of the locations are initially more attractive than others, more resources should be allocated to those with higher initial attractiveness. 

Finally, we observe that the overall pickpocket probability increases exponentially when the initial attractiveness is scaled linearly over all potential pickpocket locations. In Figure \ref{newnewnew}, we demonstrate what is needed in terms of extra resources to keep the same overall pickpocket probability (of 0.92\%). We observe an exponential behavior too. So, in order to control the exponential increase in overall pickpocket probability, also an exponential increase in the amount of resources is needed. We would like to stress that the results in this paragraph are also insightful when the number of pickpockets drops (e.g., in winter). \\

\begin{figure}[h!]
\centering
\begin{tikzpicture}[domain=1:9, scale = 1]  \begin{axis}[xmin=1, xmax=1.4, ymin=25,ymax=45,     xlabel= scaling factor $k$ for initial attractiveness, ylabel = amount of $R$ needed
]   
\addlegendentry{optimal} \addplot[black] coordinates {
    (1, 30) (1.1, 32.95) (1.2,36.2) (1.3, 39.8) (1.4, 43.75)};
\end{axis}\end{tikzpicture}\caption{Total number of resources needed to keep overall crime probability at 0.92\%.} \label{newnewnew}
\end{figure}
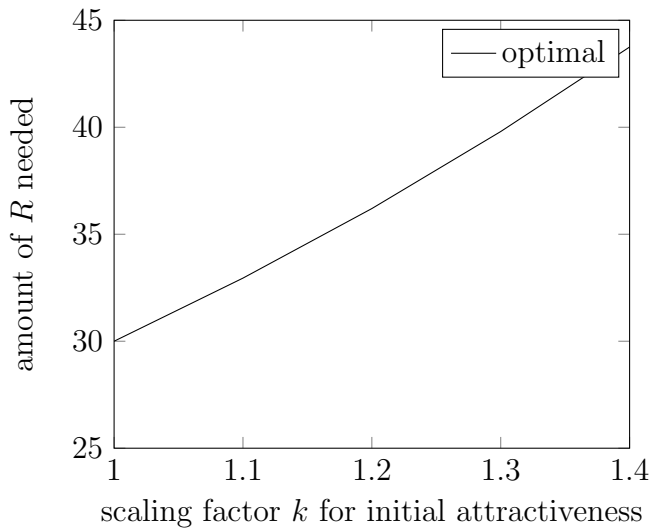

All in all, from our numerical experiments, we can formulate two main take-home messages. First, optimally allocating resources is important, especially for heterogeneous potential crime locations, which might be the case in real life situations. Second, doubling the attractiveness of locations, which may happen during special days or events, needs more than twice the number of resources needed to keep the total pickpocket probability constant.

\section{Conclusions} \label{conclu}

In this paper, we investigated how a decision maker should optimally allocate budget over protective resources and potential pickpocket locations, in order to minimize, per unit of time, the overall pickpocket probability. By using results from convex optimization, we were able to characterize the optimal budget allocation: the decision maker should first allocate the full budget over the central and local protective resources, proportionally to the sensitivity of the thief's utility towards the protective resources. Subsequently, the decision maker should allocate the budget per local protective resource over the locations based on the initial attractiveness of the locations. By means of numerical experiments, we showed that two simple intuitive rules only perform well for some specific settings. In all other cases, the performance of these simple intuitive rules is poor. Moreover, we illustrated that the extra amount of resources needed in special settings, in order to keep the total pickpocket probability stable, is growing exponential in the parameters of the situation.

In managerial terms, we thus recommend to only follow the proposed simple allocation rules if one believes that potential pickpocket locations are more or less similar in attractiveness. In all other cases, it is better to borrow insights from our optimal budget allocation rule while dividing the budget for preventing pickpocketing. We also want to emphasis the fact that to keep the probability of pickpocketing constant, significant extra protective resources might be needed during special events that increase the attractiveness of all pickpocket locations. This might be of particular importance in capacity-constrained settings.

Our study is based on several assumptions which could be relaxed in further studies. \textcolor{black}{In our current choice-based resource allocation model}, it is possible that under the optimal budget allocation some pickpocket locations face a high probability of pickpocketing, compared to others locations. In reality, such an outcome might be perceived as unfair. It would be interesting to see what happens with the total pickpocket probability whenever a fairness constraint is added. Such a fairness constraint could, for instance, reflect the idea that each location should have the same pickpocket probability.
Another direction for further research could be to investigate the stability of our results for a group of heterogeneous thieves, i.e., a group of thieves that have different tastes/interests for locations and protective resources. This will increase the realism of the setting, but also complicate a theoretical analysis of our model. In particular, we will most likely break the unimodal property of the objective function, a property that is crucial to find a closed-form expression for the minimal overall pickpocket probability. Most likely, this also applies to settings in which we do not minimize the overall pickpocket probability, but the expected impact of pickpocketing, which could be modelled as the product of the chance of pickpocket per location times its impact.

\section*{Acknowledgement}
This one is for Sien Schlicher, who hopefully can live in a safer world. Moreover, this work was partially supported by NWO grant VI.Veni.201E.017. Finally, we would like to thank the photographers, who allowed us to make use of their pictures in Figure 1.

\bibliography{references}

\end{document}